\documentclass{article}
\usepackage{amssymb,amsthm,amsmath,amsfonts}
\usepackage{epsfig}

\theoremstyle{plain}
\begingroup
 
\newtheorem{theorem}{Theorem} 
\newtheorem{corollary}{Corollary}
\newtheorem{lemma}{Lemma}
\endgroup
\newcommand{\nwc}{\newcommand}
\nwc{\Levy}{L\'{e}vy}
\nwc{\Holder}{H\"{o}lder}
\nwc{\cadlag}{c\`{a}dl\`{a}g}
\nwc{\be}{\begin{equation}}
\nwc{\ee}{\end{equation}}
\nwc{\ba}{\begin{eqnarray}}
\nwc{\ea}{\end{eqnarray}}
\nwc{\la}{\label}
\nwc{\nn}{\nonumber}
\nwc{\Z}{\mathbb{Z}}
\nwc{\C}{\mathbb{C}}
\nwc{\E}{\mathbb{E}}
\nwc{\R}{\mathbb{R}}
\nwc{\N}{\mathbb{N}}
\nwc{\prob}{\mathbb{P}}
\nwc{\Skor}{\mathbb{D}}
\nwc{\PP}{\mathcal{P}}
\nwc{\M}{\mathcal{M}}
\nwc{\law}{\stackrel{\mathcal{L}}{\rightarrow}}
\nwc{\eqd}{\stackrel{\mathcal{L}}{=}}
\nwc{\vp}{\varphi}
\nwc{\Vp}{\Phi}
\nwc{\veps}{\varepsilon}
\nwc{\eps}{\ve}
\nwc{\qref}[1]{(\ref{#1})}
\nwc{\D}{\partial}
\nwc{\dnto}{\downarrow}
\nwc{\nsup}{^{(n)}}
\nwc{\ksup}{^{(k)}}
\nwc{\jsup}{^{(j)}}
\nwc{\nksup}{^{(n_k)}}
\nwc{\inv}{^{-1}}
\nwc{\one}{\mathbf{1}}
\nwc{\argmin}{\mathrm{arg}^+\mathrm{min}}
\nwc{\argmax}{\mathrm{arg}^+\mathrm{max}}
\nwc{\Rplus}{\R_+}
\nwc{\Rorder}{\R_<}
\nwc{\xx}{\mathbf{x}}
\nwc{\emp}{\mu}
\nwc{\empN}{F^n}
\nwc{\lossN}{L^n}
\nwc{\Lip}{\mathrm{Lip}}
\nwc{\BL}{\mathrm{BL}}
\nwc{\ddm}{d}

\theoremstyle{definition}
 
\newtheorem{remark}[theorem]{Remark}

\theoremstyle{remark}
 
\numberwithin{equation}{section}
\numberwithin{figure}{section}

\begin{document}

\title{Concentration inequalities for a removal-driven thinning process\\} 
\author{Joe Klobusicky\textsuperscript{1}, Govind Menon\textsuperscript{2}}

\date{\today}

\maketitle

\begin{abstract}
We prove exponential concentration estimates and a strong law of large numbers for a particle system that is  the simplest representative of a general class of models for 2D grain boundary coarsening introduced in~\cite{JK}.  The system consists of $n$ particles in $(0,\infty)$ that move at unit speed to the left. Each time a particle hits the boundary point $0$, it is removed from the system along with a second particle chosen uniformly from the particles in $(0,\infty)$. Under the assumption that the initial empirical measure of the particle system converges weakly to a measure with density $f_0(x) \in L^1_+(0,\infty)$, the empirical measure  of the particle system at time $t$ is shown to converge to the measure with density $f(x,t)$, where $f$ is the unique solution to the kinetic equation with nonlinear boundary coupling
\be
\nn
\partial_t f (x,t) - \partial_x f(x,t) = -\frac{f(0,t)}{\int_0^\infty f(y,t)\, dy} f(x,t), \quad 0<x < \infty, 
\ee
and initial condition $f(x,0)=f_0(x)$. 

The proof relies on a concentration inequality for an urn model studied by Pittel,  and Maurey's concentration inequality for Lipschitz functions on the permutation group.

\end{abstract}



\medskip
\noindent
\footnotetext[1]
{Department of Mathematical Sciences, Rensselaer Polytechnic Institute, 110 8th Street, Troy, NY 12180
Email: klobuj@rpi.edu}

\footnotetext[2]
{Division of Applied Mathematics, Box F, Brown University, Providence, RI 02912.
Email: govind\_menon@brown.edu}

\section{Introduction}
\subsection{The kinetic equation and particle system}
An important theme in kinetic theory is to rigorously derive kinetic equations
as  hydrodynamic limits of simpler particle models. In this paper, we study the  transport equation with nonlinear boundary coupling
\ba
\label{eq:ke}
\partial_t f (x,t) - \partial_x f(x,t) &=& - \frac{f(0,t)}{M(t)} f(x,t),
\quad 0<x < \infty, \\
M(t) &=& \int_0^\infty f(x,t) \, dx,
\ea
for a positive density $f(x,t)$ with  initial condition $f(x,0)=f_0(x)$.
The associated particle system consists of $n$ particles in $(0,\infty)$
that move at unit speed to the left. Each time a particle hits the boundary
point $0$, it is removed from the system along with a second particle chosen
uniformly from the particles in $(0,\infty)$.

It is not hard to show that the kinetic equation
\qref{eq:ke} is exactly solvable. However, it is not entirely straightforward to show that the kinetic equation describes the law of large numbers for the particle system. The  difficulty is  that the time between random jumps (the `internal clock' of the system) is a deterministic function of the state immediately after each jump. The main purpose of this paper is to establish exponential concentration estimates, especially for the internal clock, that allow us to rigorously establish \qref{eq:ke} starting from the particle system.

Our particle system also has interesting connections to two discrete sampling models.
The first  model, studied in  Section~\ref{sec:urn}, is an example of a \textit{diminishing urn}.  In such a model, balls are painted one of two colors (say white and red) and placed in an urn. Balls are either removed from or added into the urn through some  predetermined
drawing rule. Typically, draws are repeated until no red balls in the urn remain. The main quantity of interest is the number of white balls left. Despite the simplicity of this model, closed-form expressions for  statistics of most diminishing urns are difficult to obtain, though several limit distributions have been obtained by generating function methods~\cite{Flajolet,Hwang,Pittel}.~\footnote{One exception, the `pills problem' posed by Knuth~and McCarthy \cite{Hesterberg,Knuth}, may be solved by a clever elementary counting argument.}  In Section~\ref{subsec:urn}, we use a recurrence relation for moment generating functions to establish asymptotic normality for the number of particles lost at time $t$ in our particle system. The proof technique follows Pittel~\cite{Pittel}. 

The second model, described in Section~\ref{sec:thinning}, is an instance of \textit{two-phase sampling}.  Particles on the positive real line are sampled without replacement from a larger collection of particles whose empirical distribution approximates
some known population density. If the particles were sampled with replacement, we are in the setting of the Glivenko-Cantelli theorem, and the DKW inequality~\cite{Dvoretzky} may be used to show that the empirical distributions converge exponentially fast to the  limit distribution. While such theorems also exist in the case without replacement (see~\cite{Saegusa}, for instance), we provide a new argument of exponential convergence using a modification of Maurey's concentration inequality for Lipschitz functions on the symmetric group. The use of Maurey's inequality in this setting is one of the main technical novelties of our work.  This proof could be of independent interest to  probabilists  interested in sampling and queueing theory.

In Sections~\ref{sec:conpoint} and~\ref{sec:unifconc}, we combine the above models to obtain exponential concentration inequalities and a completely transparent proof of convergence of the empirical distributions of the particle system to the hydrodynamic limit described by \qref{eq:ke}. The main functional law of large numbers is closely related to the
Glivenko-Cantelli theorem, and by analogy suggests a uniform central limit theorem to describe fluctuations from the hydrodynamic limit. We hope to address these issues
in future work.

\subsection{Kinetic equations for grain boundary evolution}  
The kinetic equation and limit theorem in this work were motivated by domain coarsening in two-dimensional cellular networks, in particular isotropic grain boundary networks and soap froth. A fundamental aspect of the evolution of these cellular networks, discovered by Mullins and Von Neumann~\cite{Mullins,vonNeumann},
is that the rate of change of area of an $s$-sided cell is a constant multiple
of $s-6$. Thus, the rate of change of area depends only on the topology of a cell (the number of sides), and not its geometry. Further, each cell with fewer than six sides vanishes in finite time. It follows that the kinetics of the cellular network is driven by a smooth evolution, punctuated by singular `vanishing events' when cells with positive
area gain or lose sides as a  neighboring cell shrinks to zero area. 

In the 1980s and 1990s, several physicists postulated mean-field kinetic equations to describe this process in the limit when the number of cells is large~\cite{Beenakker,Flyvbjerg,Fradkov2,Marder}. 
These models have the common form
\be
\label{eq:ke-big}
\partial_t f_s + (s-6) \partial_x f_s = \sum_{l=2}^5  (l-6) f_l(0,t) \left(
\sum_{m=2}^M A_{lm}(t) f_m(x,t)\right), \quad s=2, \ldots, M.
\ee
Here the index $s$ (for `species') describes the number of sides of the cells (its topological class), and ranges from $2$ to a maximal number $M >6$; $f_s(x,t)$ denotes the number density of $s$-sides particles with area $x$ at time $t>0$. The  common feature of these equations is that the flux into and out of
species $s$ depend on the rate at which the left-moving populations $f_l$,
$l=2,\ldots,5$, hit the origin. The matrix $A_{lm}(t)$ describes the rates at which cells switch topological class as they gain or lose edges as small cells vanish. It is obtained 
by a different {\em ad hoc\/} assumption in each work, and while each kinetic equation matches some of the experimental data, there appears to have been no side-to-side comparison of the different models.

More recently, applied mathematicians have performed extensive computational experiments on the evolution of such networks~\cite{Barmak2,Elsey2,Henseler1,Lazar2}. 
Further, there has also been some rigorous analysis of kinetic models of the type~\qref{eq:ke-big} and related stochastic models~\cite{Cohen}. In recent work~\cite{JK}, one of the authors introduced a stochastic multi-species particle system in order to obtain a rigorous foundation for \qref{eq:ke-big}. Amongst other goals, this model was introduced to evaluate the often contradictory geometric assumptions used by physicists to determine the differing right hand sides of~\qref{eq:ke-big}, in light of current computational knowledge of grain boundary evolution. 

One of the rigorous results in~\cite{JK} is a hydrodynamic limit theorem for equation \qref{eq:ke-big}. The associated particle system consists of $n$ particles, partitioned
into $n_s$ particles of each species $s$, with areas $0< x_{s,1}< x_{s,2}<
\ldots < x_{s,n_s}$. The dynamics of the system consists of pure drift --particles
of species $s$ move  with constant velocity $s-6$ -- combined with stochastic
mutations when a particle vanishes,  meaning $x_{s,1}=0$ for one of the species. 
As in the one-species particle model for the kinetic equation \qref{eq:ke}, the removal times in the multi-species models are random, but depend deterministically on the state of the system immediately after a mutation. This nontrivial coupling between mutation
and removal is the main obstruction to proofs of  hydrodynamic limit theorems. To further complicate matters, species grow at different rates,  so  that in a generic realization some particles will grow during some intervals, and shrink in others. The model \qref{eq:ke} and convergence theorems presented in this paper arose from a desire to isolate the role of particle removal in \qref{eq:ke-big}. While equation~\qref{eq:ke-big} is not exactly solvable, we hope that the use of multi-species urn models and thinning estimates as in this work provide analogous contraction estimates for \qref{eq:ke-big}. 


\subsection{Statement of results}
\subsubsection{Wellposedness of the kinetic equation}
\label{subsec:wellp}
Despite the nonlinear term on the right hand side, the kinetic equation \qref{eq:ke} is exactly solvable. Let $L^1_+$ denote the cone of non-negative functions in $L^1(0,\infty)$ equipped with the norm topology.  We first obtain a formula for classical solutions to \qref{eq:ke}. We then use this  formula to define solutions in $L^1_+$. 
\begin{theorem}
\label{thm:well-posedness}
(a) Assume $f_0 \in L^1_+\cap C^1$. There exists a unique solution to \qref{eq:ke} with $f(x,0)=f_0(x)$. The solution is given by the formula
\be
\label{eq:sol-formula}
f(x,t) = \rho(t) f_0(x+t), \quad \rho(t)  = \frac{\int_t^\infty f_0(y) \, dy}{\int_0^\infty f_0(s)\, ds}.
\ee
(b) The formula \qref{eq:sol-formula} defines a continuous dynamical system in $L^1_+$. That is the map $(t,f_0) \mapsto f(\cdot,t)$ is in $C([0,\infty)\times L^1_+, L^1_+)$.
\end{theorem}
\begin{proof}
(a) Observe that the kinetic equation \qref{eq:ke} scales like a linear equation. Therefore, without loss of generality we may assume that 
\be
\label{eq:sok-formula-corr}
M(0)= \int_0^\infty f_0(s) \, ds =1, \quad \rho(t) = \int_t^\infty f_0(s)\, ds.
\ee
We first check the solution formula under the assumption that $f_0$ is a smooth strictly positive probability density and $f(x,t)$ is given by \qref{eq:sol-formula}. Then the total number of particles is 
\be
\label{eq:number}
M(t) = \int_0^\infty f(x,t) \, dt = \rho(t) \int_0^\infty f_0(x+t) \, dx = \rho(t)^2. 
\ee
We differentiate the expression for $f(x,t)$ in equation \qref{eq:sol-formula} to find
\be
\label{eq:sol-formula2}
\partial_t f - \partial_x f = -f_0(t) f_0(x+t) \stackrel{\qref{eq:sol-formula}}{=}  -\frac{f(0,t)f(x,t)}{\rho(t)^2} = -\frac{f(0,t)}{M(t)} f(x,t).
\ee
When $f_0$ has compact support, the solution formula \qref{eq:sol-formula} continues to hold for all $t$,  and  $f(x,t) \equiv 0$ when $t \geq t_*$, where $t_* = \inf \{x \left| \int_x^\infty f_0(r)\,dr =0\right. \}$.  

(b) The main subtlety in defining solutions to \qref{eq:ke} directly with arbitrary $L^1_+$ initial data is that the pointwise boundary value $f(0,t)$ is not defined in general, even if we know that $f(\cdot,t) \in L^1_+(0,\infty)$. However, the solution formula \qref{eq:sol-formula} clearly defines a function in $L^1_+$. Further, since the shift is continuous in $L^1$ with the norm topology, and $\rho(t)$ is continuous, the solution map is continuous in $L^1_+$ with the norm topology. It is the unique extension to $L^1_+$ of the densely-defined solution map of (a).
\end{proof}
In all that follows  we will assume that the normalization \qref{eq:sok-formula-corr} holds.
We will use the following notation for  distribution functions
\be
\label{eq:def-disbn}
F(x,t) = \int_0^x f(y,t) \, dy, \quad F_0(x) = \int_0^x f_0(y) \, dy.
\ee
To fix ideas, it is useful to note the following  solution. When $f_0(x) = \one_{0 < x \leq 1}$ we find that
\be
\label{eq:sol-formula3}
M(t) = (1-t)^2, \quad f(x,t) = (1-t)\one_{0 < x < 1-t}, \quad 0 \leq t \leq 1.
\ee

Finally, let us note that the solution formula \qref{eq:sol-formula} has a formal extension to measure-valued solutions. The distribution function for an $L^1_+$ solution satisfies
\be
\label{eq:sol-formula5}
F(x,t) = \rho(t) \left(F_0(x+t) -F_0(t)\right).
\ee
This formula is meaningful when $F_0$ is an increasing \cadlag\/ function that is not necessarily continuous. However, since $t \mapsto \rho(t)$ is now discontinuous in general, the map $t \mapsto F(\cdot,t)$ does not define a continuous dynamical system. This issue is closely tied to the main well-posedness theorem of~\cite{MNP}. Our main goal in this paper is to establish a hydrodynamic limit theorem via concentration estimates and the continuity of $F_0$ plays a role in the proof. For this reason, we do not consider measure-valued solutions in this paper, though formula \qref{eq:sol-formula5} will be useful.

\subsubsection{The queueing model}
The particle system  is a queueing model defined as follows. Let $\Rorder^m$ denote the set of vectors $\xx \in \Rplus^{m}$ with $m$ strictly ordered coordinates $0< x_1 < x_2 < \ldots < x_m$. 
Each state of the particle system is a vector $x \in \Rorder^m$ for an even, positive integer $m$, and the state space is the disjoint product $E =\coprod_{m \in 2\N }\Rorder^{m}$.  The evolution of the system from an arbitrary point $\xx \in \Rplus^m$ is as follows. For $0 \leq t < \tau:=x_1$, each particle drifts to the left at unit speed,
\be
\label{eq:pdmp2}
x_i(t) = x_i - t, \quad 1 \leq i \leq m, 
\ee
until the left-most particle $x_1$ hits the origin at time $\tau$. At the hitting time, $\tau$, the particle at the origin is removed from the system, along with another particle chosen uniformly. Precisely, an index $j \in \{2,\ldots,m\}$ is chosen uniformly, and the particle $x_j(\tau)=x_j-\tau$ is removed. The vector of size $m-2$ that remains is the new state of the system. If $m \geq 4$, this process of deterministic drift followed by removal of a random particle is repeated. If not, the process terminates. It is intuitively clear that the particle system is well-defined, and it is easy to check that it satisfies the rigorous definition of a piecewise deterministic Markov process proposed by Davis~\cite{Davis}. 

We will fix a convenient initial condition  for the particle system in order to state the concentration estimates for the empirical measure. The conclusions hold in somewhat greater generality, but this family of initial conditions is natural, and allows us to convey the main ideas in a simple fashion.

Assume given an initial probability density $f_0 \in L^1_+$, and an even positive integer $n$, and recall that $F_0(x) = \int_0^x f_0(r) \, dr$ denotes the cumulative distribution function of $f_0$.  Let 
\be
\label{eq:ic}
a_k = F_0^{-1}\left(\frac{2k-1}{2n}\right), \quad 1 \leq k \leq n.
\ee
We assume that the $n$-particle system starts at the state  $\xx(0) =(a_1,a_2, \ldots,a_n)$.
The state of the system becomes random after time $\tau_1=x_1(0)$, when the leftmost particle hits the origin. The hitting times are denoted $\tau_1 < \tau_2 < \ldots < \tau_{n/2}$. The state of the system is a \cadlag\/ path $\xx(t)$ in $E$ which jumps at the times $\tau_k$, $1 \leq k \leq n/2$.  We let $\prob_n$ denote the law of the process $\xx(t)$. When $n$ is fixed, we simply write $\prob$. 

\subsubsection{Concentration estimates}
We keep track of the loss of particles at the origin through the distribution function 
\be
\label{eq:loss-defn}
\lossN(t) = \frac{1}{n} \sum_{i=1}^{n/2} \one_{t \geq \tau_i}.
\ee

Loosely speaking, $\lossN(t)$ is the `internal clock' of the system. The main subtlety in the problem is that the number of jumps before a tagged particle is removed from the system -- either because it hits the origin, or because it is randomly chosen for deletion -- is random. However, we expect that in the $n \to \infty$ limit, the rate of loss will be determined by the boundary value $f(0,t)$. In order to express a law of large numbers for $\lossN$, we define the limiting loss distribution function $L(t)$ for equation \qref{eq:ke} by the conservation law
\be
\label{eq:Ldef}
2L(t) + M(t) = M(0) =1, \quad\mathrm{or}\quad L(t) =  \frac{1}{2}\left( 1- \rho(t)^2 \right). 
\ee
The factor of $2$ reflects the fact that two  particles are lost each time a particle hits the origin. Observe that $L(t)$ is continuous in time, because of our assumption that $f_0 \in L^1_+$. Continuity is used in the proof of the  following uniform concentration estimate for $\lossN$.
\begin{theorem}
\label{thm:conc-L}
For every $\veps>0$ there exists an $n_\veps$ such that for $n\geq n_\veps$
\be
\label{eq:conc-L}
\prob_n\left( \sup_{t \in [0,\infty)} \left| \lossN(t) - L(t) \right| > \frac{\veps}{2} \right) \leq \frac{2}{\veps}  \, e^{-8n\veps^2}.
\ee
The parameter $n_\veps$ is given implicitly by $n_\veps\veps = 4C\log n_\veps$ where $C$ is a universal constant.
\end{theorem}

\subsubsection{Concentration of the empirical measure}
The cone of positive measures on $\R_+$ is denoted by $\M$. 
The duality pairing between $\mu \in \M$ and a continuous function $\varphi \in C(\R_+)$ is expressed as
\be
\label{eq:def-dual}
\langle \mu, \varphi \rangle = \int_0^\infty \varphi (x) d\mu(x).
\ee
The space $\M$ equipped with the weak-* topology  may be metrized using the space of bounded Lipschitz functions~\cite{Dudley-RAP},
\begin{align}
BL(\mathbb{R}^+) = \{\varphi \in C(\mathbb R^+):\|\varphi\|_{\BL}<\infty\}, \\
\|\varphi\|_{\BL} = \sup_{x} |\varphi(x)| + \sup_{x,y} \frac{|\varphi(x)-\varphi(y)|}{|x-y|}.
\end{align}
The distance between two measures $\mu,\nu \in \M$ in the  BL-metric is 
\begin{equation}
\ddm(\mu, \nu) = \sup_{\|\varphi\|_{\BL} \leq 1} \langle \mu - \nu, \varphi \rangle.
\end{equation}
The empirical measure defined by  each state $\xx(t) \in \R_+^m$, and the empirical measured defined by $\xx(0) =(a_1,a_2,\ldots,a_n)$ are denoted
\be
\label{eq:def-emp}
\empN(t) = \frac{1}{n} \sum_{i=1}^m \delta_{x_i(t)}, \quad \empN_0 = \frac{1}{n}\sum_{i=1}^n \delta_{a_i}.
\ee
Finally, the following notation is convenient.  For each $h >0$, we define the shift operator $S_h$ acting on bounded, measurable functions, and its dual operator $S_h^*$ acting on measures, as follows
\be
\label{eq:shift1}
\left(S_h \varphi\right) (x) = \varphi(x-h) \one_{x \geq h}, \quad \left(S_h^*\right) \mu(x) = \mu(x+h) -\mu(h), \quad x \in (0,\infty).
\ee
(Here and in what follows, we use the same notation for a measure and its \cadlag\/ distribution function, $\mu(x)= \mu([0,x))$ when there is no possibility of confusion).
The solution formula \qref{eq:sol-formula5} may be expressed in terms of the shift map
as $F(\cdot,t)= \rho(t) S_t^*F_0(\cdot)$.

\begin{theorem}
\label{thm:conc-emp}
There exist universal constants $K,\kappa>0$, such that for every $\veps>0$ and $T>0$ there exists $n_\veps$, $M_\veps$, $N_\veps$ such that
\be
\label{eq:conc-eps}
\prob_n\left(  \sup_{t \in [0,T]} \ddm\left(\empN(t), \rho(t)S_t^*F_0\right) > \veps \right) \leq K \left(M_\veps N_\veps + \frac{1}{\veps} \right) e^{-\kappa n \veps^2}.
\ee
\end{theorem}
Let $\mathbb{Q}$ denote the product measure $\prod_{n=2}^\infty \prob_n$. By the Borel-Cantelli lemma, we obtain a strong law of large numbers.
\begin{corollary}
\label{cor:SLLN}
For every $T>0$,
\be
\lim_{n \to \infty} \sup_{t \in [0,T]} \ddm\left(\empN(t), \rho(t)S_t^*F_0\right)       =0, \quad \mathbb{Q} \; a.s.
\ee
\end{corollary}
The $\veps$ dependence on the parameters in the theorem is as follows. First, $n_\veps$ must be chosen so that the condition of Theorem~\ref{thm:conc-L} holds, and so that the distance $\ddm(F_0,\empN_0)$ between the initial empirical measure, $\empN_0$, and the data, $F_0$, is  $O(\veps)$ for $n \geq n_\veps$. The parameter $M_\veps$ depends on the tail of the initial data $F_0$, but not on $T$. Given $\veps>0$, let $x_*$ be chosen so that $1-F_0(x) < \veps$ for $x > x_*$. The space of bounded Lipschitz functions on $[0,x_*]$ is totally bounded, and $M_\veps$ is the smallest number of $\veps$-balls required to cover the space $\BL([0,x_*])$. This number may be estimated using the Kolmogorov-Tikhomirov estimate~\cite{Kolmogorov}. The parameter $N_\veps$ is related to the modulus of continuity of $F_0$ and depends on $T$. Given $\veps>0$, let $h$ be chosen so that $\sup_{x} F_0(x+h)-F_0(x) < \veps$. Then $N_\veps=T/h$.

\subsubsection{Outline of the paper}
We establish Theorem~\ref{thm:conc-L} by first studying the combinatorics of a `diminishing urn' model in Section~\ref{sec:urn}. Once Theorem~\ref{thm:conc-L} has been established, we introduce a second simplified model -- uniform thinning of a finite point process -- and establish a concentration inequality for this process  using  Maurey's concentration inequality for the permutation group. We then combine these estimates with some simple regularity estimates for the empirical measure, $\empN(t)$, to establish Theorem~\ref{thm:conc-emp}.

\section{The concentration estimate for $\lossN(t)$} 
\label{sec:urn}
Recall that the initial data for the particle system is the state $\xx=(a_1,\ldots,a_N)$ defined in equation \qref{eq:ic}. Given $t>0$, suppose $w$ is the largest integer such that $a_w \leq  t$. Since all particles move to the left at unit velocity, the particles $a_1,\ldots, a_w$ are all removed from the system by time $t$. However, these particles could be removed either because they hit the origin, or because they were randomly selected. The loss measure $n\lossN(t)$ counts only the particles that hit the origin. Thus, in order to estimate it, we must distinguish between the two possibilities for removing particles. The combinatorics of this process does not depend on the spatial arrangement of the points.  In fact, a closely related process appeared as a model of canibbalistic behavior in a population, and was analyzed by Pittel~\cite{Pittel}. We follow his work in the next subsection.

\subsection{The diminishing urn}
\label{subsec:urn}
Let $r \leq n$ be positive integers. Consider an urn with $w$ white balls and $r =n- w$ red balls. Balls are removed randomly, with a draw occurring in the following way. First, a white ball is removed from the urn.  Next, another ball is chosen randomly from the remaining balls. This process ends when all the white balls have been removed.~\footnote{In  Pittel's model either  one white ball is  removed, or two  white balls are removed and  one red ball is added.} Our interest lies in the quantity $d_{n,r}$, the total number of draws, and $X_{n,r}$, the terminal number of red balls. We will prove results about $X_{n,r}$. These are equivalent to results about $d_{n,r}$. Indeed, given $X_{n,r}$, the total number of balls removed from the urn is $n-X_{n,r}$, and since two balls are removed at each  draw, the total number of draws is 
\begin{equation}
\label{eq:d-translate}	
d_{n,r} = \frac{n-X_{n,r}}2.
\end{equation}

In order to state the limiting law for $X_{n,r}$ we define the functions
\be
\label{eq:def-phi}
\phi(x) = x^2,\quad\mathrm{and}\quad \psi(x) = 2x^2(1-x)^2, \quad x \in [0,1].
\ee
\begin{theorem}
\label{thm:urn}	
(a) Assume $\lim_{n \to \infty} r/n=\rho \in (0,1)$. Then the random variables
\be
\frac{X_{n,r} - n\phi(\rho)}{(n\psi(\rho))^{\frac {1}{2}}} \quad\mathrm{and}\quad
\frac{2d_{n,r} - n(1-\phi(\rho))}{(n\psi(\rho))^{\frac {1}{2}}}
\ee
converge in distribution to the standard normal law. 

(b) For every $\veps>0$ there exists $n_\veps>0$ such that for all positive integers $n$ and $r$ with $r/n=\rho \in (0,1)$ and $n\geq n_\veps$
\ba
\label{eq:conc-X}
&& \mathbb{P}\left( \left|\frac{X_{n,r}}n-\phi(\rho) \right |>\veps \right ) \leq 2\exp
\left (-\frac{n \veps^2}{4\psi(\rho)} \right), \\
\label{eq:dnr}
&& \mathbb P\left(\left|\frac{d_{n,r}}{n} - \frac{1}{2}\left(1-\phi(\rho)\right) \right|> \veps\right) \leq 2\exp\left(-\frac{n\veps^2}{\psi(\rho)} \right). 
\ea
The parameter $n_\veps$ is given implicitly by $n_\veps \veps^2 = 4C \veps \log n_\veps$ where $C>0$ is a universal constant.
\end{theorem}
The theorem is proved by computing the asymptotics of the Laplace transform of the law of $X_{n,r}$, given by
\begin{equation}
\label{eq:LT}
f_{n,r}(z) = \mathbb{E}\left[\exp(zX_{n,r})\right], \quad -\infty<z<\infty.
\end{equation}  
If the number of white ball is zero or one, we find
\begin{align}
f_{n,n}(z) = e^{nz}, \quad n\ge 1,\\
f_{n,n-1}(z) = e^{(n-1)z} \quad n\ge 2.
\end{align}
In the general case, we use the Markov property of $X_{n,r}$ to obtain the recurrence relation
\begin{equation}\label{eq:recrel}
f_{n,r}(z) = \begin{cases}(1-\frac r{n-1})f_{n-2,r}+\frac r{n-1}f_{n-2, r-1} & 0 \le r \le n-2,  \\
f_{n-1,r-1} & r = n-1. \\
\end{cases}
\end{equation}
This relation can be expressed compactly in terms of a linear operator 
\begin{equation}
\label{eq:rec-op}	
 f_{n,r}(z) = T_{nr}\left[f_{n-2,\cdot}(z)\right], \quad 0\le r\le n-1, \quad n\ge 2.
\end{equation}
We will show below that the leading order asymptotics of $f_{n,r}(z)$ as $n \to \infty$ is given by the Laplace transform of a normal random variable
\begin{equation}
\label{eq:g-ansatz}	
g_{n,r}(z) = \exp\left(zn\phi(\rho) +\frac{z^2}2n\psi(\rho)\right), \quad \rho = \frac{r}{n}.
\end{equation}
We assume for now that $\phi$ and $\psi$ are unknown -- equation \qref{eq:def-phi} follows  from substituting the ansatz \qref{eq:g-ansatz} in \qref{eq:recrel} and evaluating the leading order terms. To this end, observe that 
\begin{equation}
g_{n-2,r}(z) = \exp\left( z(n-2)\phi\left(\frac r{n-2}\right)+\frac{z^2}2(n-2)\psi\left(\frac r{n-2}\right)\right).
\end{equation}
Therefore, by elementary algebra 
\begin{align}\label{lone}
L_1:= \log \left( \frac{g_{n-2,r}(z)}{g_{n,r}(z)}\right) = zn\left[\phi\left(\frac r{n-2}\right)-\phi\left(\frac r{n}\right)\right]-2z\phi\left(\frac r{n-2}\right) \\ +\frac{z^2}2n\left[\psi\left(\frac r{n-2}\right)-\psi\left(\frac r{n}\right)\right]-\frac{z^2}2\cdot2\psi\left(\frac r{n-2}\right). 
\end{align}
Similarly,
\begin{align}\label{ltwo}
L_2: = \log \left( \frac{g_{n-2,r-1}(z)}{g_{n,r}(z)}\right) = zn\left[\frac{n-2}{n}\phi\left(\frac {r-1}{n-2}\right)-\phi\left(\frac r{n}\right)\right] \\ +\frac{z^{2}}2n\left[\frac{n-2}{n}\psi\left(\frac {r-1}{n-2}\right)-\psi\left(\frac r{n}\right)\right]  
\end{align}
The first-order asymptotics of the ratios in the arguments of $\phi$ and $\psi$ are clearly
\be
\frac{r}{n-2} = \rho\left(1+\frac 2n \right )+ O\left(\frac 1{n^2}\right), \quad \frac{r-1}{n-2}=\rho+\frac {2\rho -1}n +  O\left(\frac 1{n^2}\right).
\ee
We use the above expressions, the arguments of $\phi$, and Taylor's theorem to obtain
\ba
\label{eq:err-phi1}
\phi\left(\frac{r}{n-2}\right) &=&  \phi(\rho)+\frac 2n \rho \phi'(\rho)+  O\left(\frac 1{n^2}\right), \\
\phi\left(\frac{r-1}{n-2}\right) &=& \phi(\rho)+\frac {2\rho-1}n  \phi'(\rho)+O\left(\frac 1{n^2}\right).
\ea
More precisely, the error terms above satisfy 
\be
\label{eq:err-control}
\left| O\left(\frac 1{n^2}\right) \right| \leq \frac{C}{n^2} \max_{\rho \in (0,1)} |\phi''(\rho)|.
\ee
Similar expansions for $\psi(x)$ may be applied to (\ref{lone}) and (\ref{ltwo}), with the  error terms again dominated by $\max_{\rho \in (0,1)}|\psi''(\rho)|$. We substitute the expansion for $\phi$ and $\psi$ into equations \qref{lone}  to obtain
\be
\nn
L_1=  z(2\rho\phi'(\rho)-2\rho)+ \frac{z^2}2(2\rho\psi'(\rho)+2\rho)+O\left(\frac {|z|}{n}\right)
= A z+C\frac{z^2}2+ O\left(\frac {|z|}{n}\right).
\ee
Similarly,
\begin{align}
\nn
L_2 = z[(2\rho-1)(\phi'(\rho)-2\phi(\rho))]+\frac{z^2}2[(2\rho-1)(\psi'(\rho)-2\psi(\rho))]+O\left(\frac {|z|}{n}\right)\\
\nn
:= Bz +D\frac{z^2}2+O\left(\frac {|z|}{n}\right).
\end{align}
We use the above expressions and equations \qref{lone} and \qref{ltwo} to find
\begin{align}
 \frac{g_{n-2,r}(z)}{g_{n,r}(z)} = 1+A z+(A^{2}+C)\frac{z^2}2+ O\left(\frac {|z|}{n}+|z|^3\right),\\
\frac{g_{n-2,r-1}(z)}{g_{n,r}(z)} = 1+B z+(B^{2}+D)\frac{z^2}2+ O\left(\frac {|z|}{n}+|z|^3\right).
\end{align}
We now use the recurrence relation (\ref{eq:recrel}) to obtain
\ba
\nn
\lefteqn{\frac{T_{n,r}[g_{m-2,\cdot}](z)}{g_{m,r}(z)}= 1+z\left[\rho B+(1-\rho)A\right]} \\
\nn
&&
+\frac{z^2}2[\rho(B^2+D)+(1-\rho)(A^2+C)]+O\left(\frac {|z|}{n} + |z|^3\right).
\ea
The coefficient of $z$ vanishes if the following differential equation holds,
\begin{equation}
0 = \rho B+(1-\rho)A =\rho\phi'(\rho)-2\phi(\rho), \quad \phi(1) = 1.\\
\end{equation}
(The initial condition is determined by the extreme case when $r=n$). We thus find $\phi(\rho) = \rho^2$ as in \qref{eq:def-phi}. Similarly, the coefficient of $z^2$ vanishes if the following differential equation is satisfied,
\be
0 =\rho(B^2+D)+(1-\rho)(A^2+C) =  \rho \psi'(\rho)-2\psi(\rho)+4\rho^3(1-\rho).
\ee
Again the condition $\psi(1)=0$ follows from the extreme case when $r=n$. By direct solution, or inspection, we see that this equation also has the polynomial solution $\psi(\rho) = 2\rho^2(1-\rho)^2$. This establishes \qref{eq:def-phi}.

Since $\max_{\rho \in [0,1]} |\phi''(\rho)|$ and $\max_{\rho \in [0,1]} |\psi''(\rho)|$ are bounded by universal constants, the error terms are uniformly controlled if the domain of $z$ is suitably restricted. We state these results as in ~\cite[Lemma 1]{Pittel} 
\begin{lemma}\label{glem}
Fix $u>0$ and consider $z$ such that $|z|\sqrt{n} \leq u$. Then uniformly over $0 \leq 2 \leq m \leq n$ and $0 \leq r \leq m$,
\begin{equation}
T_{m,r}[g_{m-2,\cdot}(z)] = g_{m,r}(z) \exp(O(|z|m^{-1})).
\end{equation}
\end{lemma}
We apply this lemma and sum over the errors incurred as $m$ increases from $2$ to $n$, to obtain the following estimate.
\begin{lemma}\label{fglem}
Under the assumptions of Lemma~\ref{glem}, the following estimates holds uniformly over $2\le m\le n$ and $0\le r \le m$,
\begin{equation}
f_{m,r}(z) = g_{m,r}(z) \exp\left[O (|z| \log(m))\right].
\end{equation}
\end{lemma}
The proof of both the Lemmas and the first assertion in Theorem~\ref{thm:urn} is identical to that in~\cite{Pittel}. The only substantial difference here is the  computation of the functions $\phi(\rho)$ and $\psi(\rho)$. For these reasons we refer the reader to~\cite{Pittel} for these proofs.

The concentration inequality \qref{eq:conc-X} follows easily from Lemma~\ref{fglem}.  To see this, rewrite Lemma \ref{fglem} as 
\begin{equation}
f_{m,r}(\frac{u}{\sqrt n})= \mathbb{E}\left[\exp \left(\frac{u}{\sqrt n}X_{n,r}\right )\right] \le g_{n,r}\left( \frac{u}{\sqrt n}\right) \exp\left(\frac{C |u|\log(n)}{\sqrt n}\right)
\end{equation}
where  $ u \in \mathbb{R}$ and $C>0$ is a sufficiently large constant. This immediately implies 
\begin{equation}\label{evest}
\mathbb{E}\left[\exp \left(\frac{u}{\sqrt n}(X_{n,r}- n\phi(\rho))\right )\right] \le \exp\left (\frac{u^2}2 \psi(\rho) \right)\exp\left (\frac{C|u|}{\sqrt n} \log n \right).
\end{equation}
The concentration inequality \qref{eq:conc-X} is now obtained as follows. For brevity, let 
\begin{equation}
Y = \frac{X_{n,r} - n\phi(\rho)}{\sqrt n}, \quad\mathrm{and\;fix}\quad a>0.
\end{equation}
Then for any $u>0$, by Chebyshev's inequality
\be
\mathbb{P}(Y>a) \le \mathbb E\left[e^{u(Y-a)} \mathbf{1}_{Y>a}\right]
\le e^{-au}\exp\left (\frac{u^2}2 \psi(\rho) \right)\exp\left (\frac{C|u|}{\sqrt n} \log n \right).
\ee
We choose $u$ to minimize the product of the first two terms of this expression (the last term is asymptotically negligible). This yields $u = a/ \psi(\rho)$ and 
\begin{equation}
\mathbb{P}(Y>a) \le \exp\left (-\frac{a^2}{2 \psi(\rho)} \right)\exp\left (\frac{Ca}{\psi(\rho)\sqrt n} \log n \right).
\end{equation}
A similar estimate for $\mathbb{P}(Y<-a)$ is obtained by essentially the same  calculation. 
Finally, writing $a = \veps \sqrt n$, we obtain 
\be
\label{eq:conc-X-mod}
\mathbb{P}\left( \left|\frac{X_{n,r}}n-\phi(\rho) \right |>\veps \right ) \leq 2\exp
\left (-\frac{1}{2\psi(\rho)}\left( n \veps^2 -2C\veps \log n\right) \right).
\ee
Let $n_\veps$ be defined as in the statement of Theorem~\ref{thm:urn}. Then for $n \geq n_\veps$, 
\be
n \veps^2 -2C\veps \log n \geq \frac{1}{2} n \veps^2,
\ee
and \qref{eq:conc-X} follows. The assertions about $d_{n,r}$ in Theorem~\ref{thm:urn} follow from  \qref{eq:d-translate}  and \qref{eq:conc-X}. This completes the proof of Theorem~\ref{thm:urn}.

\subsection{From $d_{n,r}$ to $\lossN(t)$}
\label{subsec:dtoL}
We now return to the queueing model. Fix $t>0$, let $w$ be the largest integer such that $a_w \leq t$, and $r=n-w$. Color the particles $a_1 \ldots a_w$ white, and the particles $a_{w+1}, \ldots, a_n$ red. Since $a_w \leq t < a_{w+1}$ and all particles move to the left at unit speed, by time $t$ all the white particles have been removed. Further, any red particles lost have only been removed by random selection, not by hitting the origin. Thus, the random variable $n\lossN(t)$ which counts the number of particles removed at the origin in the queueing model, is exactly the same as the number of draws $d_{n,r}$ in the urn model. Thus, $n\lossN(t)$ has the same distribution as $d_{n,r}$. As we let $n\to \infty$ with $t>0$ fixed, 
\be
\label{eq:rho-L}
\rho =\rho(t)= \lim_{n\to \infty} \frac{r}{n} = 1 -F_0(t) = \int_t^\infty f_0(r)\, dr. 
\ee
Therefore, using the identities \qref{eq:Ldef} and \qref{eq:def-phi}, the expected value of $\lossN(t)$ as $n \to \infty$ is
\be
\frac{1}{2}(1- \phi(t)) = \frac{1}{2} \left(1 - \left(1-F_0(t)\right)^2\right) = L(t).
\ee
Since $d_{n,r}$ has the same law as $n\lossN(t)$, we see that \qref{eq:dnr} is equivalent to the following concentration estimate 
\begin{equation}\label{eq:pointest}
\mathbb P\left(\left|L^n(t) - L(t)\right|> \veps \right ) \le 2\exp\left(-\frac{n \veps^2}{\psi(\rho)}\right) \leq 2 e^{-8n\veps^2}, \quad n \geq n_\veps,
\end{equation}
since $\max_{\rho \in [0,1]} \psi(\rho) = 1/8$.

\subsection{Proof of Theorem~\ref{thm:conc-L}} 
We now use the pointwise estimate \qref{eq:pointest} to obtain a uniform estimate over the interval $t \in [0,\infty)$. 

We define a partition $0 = t_0 < t_1 < \dots < t_{n-1} < t_{P}$  of the interval $[0,\infty)$ as follows. We set $t_0 = 0$ and  
\be
\label{eq:deft}
t_{i+1} =  \inf_{s >t_i} \{s:L(s)-L(t_i) \ge \varepsilon/2\}.
\ee
These points are well-defined and strictly increasing because as equation \qref{eq:Ldef} shows, $L$ is a positive, continuous, increasing function with limit $L(\infty) = 1/2$. Since $L$ is increasing, it is immediate that 
\be
\label{eq:propL}
0 \leq L(t)-L(t_i) \le \frac{\varepsilon}{2}, \quad t \in [t_i,t_{i+1}), \quad i = 0, \dots, P-1,
\ee
and $P \leq \veps^{-1}$ since $L(\infty)=1/2$. 

The difference $\lossN(t) -L(t)$ at an arbitrary point $t \in [0,\infty)$ can be controlled using estimate \qref{eq:pointest} at the endpoints $\{t_i\}_{i=0}^P$.  Each point $t \in [0,\infty)$ lies in a unique interval $[t_j, t_{j+1})$ for some $j \in \{0, \dots, P\}$, where we denote $t_{P+1} = \infty$. Since both $\lossN$ and $L$ are increasing, \cadlag\/  functions    
\begin{align}
\nn
\lossN(t) -L(t) \le L^n(t_{j+1})-L(t_j) = \left(L^n(t_{j+1})-L(t_{j+1})\right)+\left(L(t_{j+1})-L(t_j)\right) \\
\le \max_{1\leq i\leq P} \left|\lossN (t_{i})- L(t_{i})\right| + \frac \varepsilon 2.
\end{align}
(We have shifted the index on the second term, and used the fact that $\lossN(\infty)=L(\infty)=1/2$.)  Similarly, it follows that for each $t \in [0,\infty)$
\begin{align}
\nn
L(t) -\lossN(t) \le L(t_{j+1})-\lossN(t_j) = \left(L(t_{j+1})-L(t_{j})\right)+\left(L(t_{j})-\lossN(t_j)\right) \\
\le  \frac \varepsilon 2 + \max_{1 \leq i \leq P} \left|\lossN (t_i)- L(t_i)\right| .
\end{align}
(The lower index is $1$ because $\lossN(t_0)=L(t_0)=0$). Since the above estimate is uniform in $t$, 
\begin{equation}
\sup_{t \in [0,\infty)} \left|\lossN(t) - L(t)\right| \le \frac{\varepsilon}{2} + \max_{1 \leq i \leq P} \left|\lossN(t_i)- L(t_i)\right|.
\end{equation}
We then use our pointwise concentration estimate (\ref{eq:pointest}) to obtain 
\ba
\nn
\lefteqn{	
\mathbb{P}\left(\sup_{t \in [0,\infty)} \left|\lossN(t)- L(t)\right| \ge \varepsilon\right)
\le \mathbb{P}\left(\max_{1 \leq i\leq P} \left|\lossN(t_i)- L(t_i)\right|\ge \frac{\varepsilon}{2}\right) }\\
&&\nn
\le \sum_{i = 1}^{P}\mathbb{P}\left(\left|\lossN(t_i)- L(t_i)\right|\ge \frac{\varepsilon}2\right)  \stackrel{\qref{eq:pointest}}{\le} \frac{2}{\veps} \max_{1 \leq i \leq P} \exp\left(-\frac{n\veps^2}{\psi(\rho(t_i)}\right)
\leq \frac{2}{\veps}  e^{-8n \veps^2}.
\ea
In the last step, we have used the fact that $P \leq \veps^{-1}$; chosen $n_\veps$ so that $C\log n_\veps/n_\veps = \veps$ and assumed $n \geq n_\veps$; and replaced $\psi(\rho(t_i))$ by the uniform upper bound  $\max_{\rho \in [0,1]} \psi(\rho) = \max_{\rho \in [0,1]} 2 \rho^2 (1-\rho)^2 =1/8$. This completes the proof of Theorem~\ref{thm:conc-L}.

\section{A concentration inequality for uniform thinning}
\label{sec:thinning}
The dynamics of the queueing model consists of translation and thinning. In this section we prove a concentration inequality for the thinning of a point-set. As in Section~\ref{sec:urn} the result is stated in a manner that is independent of the queueing model.

Assume given a set of $r$ points $b_1 < b_2 < \ldots < b_r$ on $\R_+$. We thin this set by choosing a subset $b_{j_1}, b_{j_2}, \ldots, b_{j_s}$  of size $s \leq r$ from this set of particles uniformly. We denote the empirical measure of the full and thinned subsets by
\be
\label{eq:def_mu}
\nu = \frac{1}{r}\sum_{i=1}^r \delta_{b_j}, \qquad  \mu = \frac{1}{r} \sum_{i=1}^s \delta_{b_{j_i}}.
\ee
Let $\mathcal{E}$ denote the set of empirical measures $\mu$ obtained by thinning as above.
There are $r\choose s$ distinct thinned subsets; thus $|\mathcal{E}|= {r\choose s}$. Let $\prob_{r,s}$ denote the uniform probability measure on $\mathcal{E}$. 
\begin{theorem}
\label{thm:thin}	
Assume $\varphi: \R_+ \to \R$ is a bounded measurable function. For every $\veps>0$
\be
\label{eq:conc-thin}
\prob_{r,s} \left( \mu: \left| \langle \mu, \varphi \rangle - \frac{s}{r} \langle \nu, \varphi \rangle \right | > \veps\right) \leq 2 \exp 
\left( -\frac{r \veps^2}{64 \|\varphi\|_\infty^2} \right).
\ee
\end{theorem}
\begin{corollary}
\label{cor:thin}	
For every $\veps>0$, there exists a positive integer $M(\veps,\nu)>0$ such that 
\be
\label{eq:conc-thin-dist}
\prob_{r,s} \left( \mu: \ddm\left(\mu, \frac{s}{r}\nu\right) > 2\veps \right) \leq 2M \exp \left( -\frac{r \veps^2}{64} \right).
\ee
\end{corollary}
\begin{proof}[Proof of Theorem~\ref{thm:thin}]
{\em 1.\/} This theorem is a direct consequence of Maurey's concentration inequality on the permutation group, once it has been suitably reformulated. To this end,
let $S_m$ denote the permutation group acting on $m$ elements. We first show that $\mathcal{E}$ is in bijection with the quotient space $C_{r,s}= S_r/(S_s\times S_{r-s})$. To construct this  bjiection, we associate to each $\pi \in S_r$ the empirical measure 
\be
\label{eq:mu-pi}
\mu(\pi) = \frac{1}{r} \sum_{i=1}^s  \delta_{b_{\pi_i}}.
\ee
Since $\mu(\pi)$ depends only on the first $s$ elements of $\pi$, the map $\pi \mapsto \mu(\pi)$ is invariant under the action of $S_{r-s}$ on the last $r-s$ elements of $\pi$. It is also clear from equation \qref{eq:mu-pi} that $\pi \mapsto \mu(\pi)$ is invariant under the action of $S_s$ on the first $s$ elements of $\pi$. 

\medskip
\noindent
{\em 2.\/} We equip $S_r$ with the normalized Hamming distance
\begin{equation}
d_H(\pi, \tau) = \frac{1}{r} \sum_{i = 1}^r \mathbf 1_{\{\pi_i \neq \tau_i\}}, \quad \pi, \tau \in S_r.
\end{equation}
Each bounded Lipschitz function $\varphi: \R_+ \to \R$ may be lifted into a Lipschitz map $f_\varphi: S_r \to \R$ by setting
\be
\label{eq:bl}
f_\varphi(\pi) = \langle \mu(\pi), \varphi \rangle = \frac{1}{r} \sum_{i=1}^s \varphi(b_{\pi_i}). 
\ee
Suppose two permutations $\pi,\tau \in S_r$ differ at $p$ indices $(i_1, \dots, i_p)$.  Then $d_H(\pi, \tau) = p/r$, and we find 
\be
|f_{\varphi}(\pi)-f_{\varphi}(\tau)| \le \frac{1}{r} \sum_{k = 1}^p |\varphi(x_{\pi_{i_k}}) - \varphi(x_{\tau_{i_k}})| \leq 2\|\varphi\|_\infty d_H(\pi,\tau). 
\ee
Thus, the Lipschitz constant of $f_\varphi$ is not larger than $2\|\varphi\|_\infty$.  (Observe that we did not need to assume that $\varphi$ is bounded and Lipschitz; $f_\varphi$ defines a Lipschitz function on $(S_r,d_H)$ provided $\varphi$ is bounded and measurable.)

\medskip
\noindent
{\em 3.\/} The following concentration inequality holds on $(S_r,d_H)$~(see the discussion on \Levy\/ families in \S 6.5 and \S 7.6  in~\cite{Milman}; the result first appears in~\cite{Maurey}). 
\begin{theorem}\textbf{Maurey's inequality.} Let $f$ be an $M$-Lipschitz function on $(S_r,d_H)$, and $\mathbb Q_r$ be the uniform measure on $S_r$.  Then for any $\veps \geq 0$
\begin{equation}\label{maurey}
\mathbb Q_r(\pi : |f(\pi) - \mathbb E_{\mathbb Q_r}\left(f\right)|>\veps) \le 2\exp\left({\frac{-r\veps^2}{16M^2}}\right).
\end{equation}
\end{theorem}

\medskip
\noindent
{\em 4.\/} We apply Maurey's inequality to $(\mathcal{E},\prob_{r,s})$ as follows. Given $\mu \in \mathcal{E}$ and $\varphi$, by the construction above $\langle \mu, \varphi \rangle = f_\varphi(\pi)$ for each $\pi \in S_r$ such that $\mu(\pi)=\pi$. Thus, lifting expectations over $\prob_{r,s}$ into expectations over $\mathbb{Q}_{r,s}$ by summing over all $\pi$ in the equivalence class of $\mu$, we find that 
\be
\label{eq:average}
\E_{\mathbb{Q}_{r,s}} f_\varphi(\pi) = \E_{\prob_{r,s}} \langle \mu, \varphi \rangle = \frac{s}{r} \langle \nu, \varphi \rangle.
\ee
Similarly, the measure of the set on which deviations are larger than $\veps$ is identical. That is,
\be
\label{eq:lift}
\prob_{r,s} \left( \mu: \left| \langle \mu, \varphi \rangle - \frac{s}{r} \langle \nu, \varphi \rangle \right | > \veps\right) =  \mathbb Q_r(\pi : |f_\varphi(\pi) - \mathbb E_{\mathbb Q_r}\left(f_\varphi\right)|> \veps).
\ee
Theorem~\ref{thm:thin} now follows from Maurey's inequality and the fact that $f_\varphi$ is $2\|\varphi\|_\infty$-Lipschitz.
\end{proof}

\begin{proof}[Proof of Corollary~\ref{cor:thin}]
Assume that $x_*>0$ is chosen so $\nu([x_*,\infty))<\veps/2$. Since $\mu$ is obtained by thinning $\nu$, this estimate also holds for $\mu$. 


Since the space of bounded Lipshitz functions on any finite interval is totally bounded, there exists an integer $M(\veps)$ and a set of functions $\varphi_i \in BL(\mathbb R_+), i= 1, \dots, M(\veps)$ with $\|\varphi\|_{\BL} \leq 1$,  such that
\begin{equation}
\ddm(\mu, \frac{s}{r}\nu) \le \max_{i \in 1, \dots, M} \langle \mu - \frac{s}{r}\nu, \varphi_i \rangle+\varepsilon.
\end{equation} 
Therefore, 
\begin{align}
\nn	
\mathbb{P}_{r,s}\left(\mu: \ddm (\mu, \frac{s}{r}\nu) \ge 2 \varepsilon\right) \le \mathbb{P}_{r,s}\left(\max_{1 \leq i\leq M(\veps,\nu)} \left| \langle \mu - \frac{s}{r}\nu, \varphi_i \rangle \right| \ge \varepsilon \right)\\ 
\nn
\le \sum_{i = 1}^{M(\veps)} \mathbb{P}_{r,s} \left( \left|\langle \mu - \frac{s}{r}\nu, \varphi_i \rangle  \right| \ge \varepsilon \right),
\end{align}
and \qref{eq:conc-thin-dist} follows by an application of Theorem~\ref{thm:thin}.
\end{proof}
\begin{remark}
\label{rem:uniformM}	
Observe that the constant $M(\veps,\nu)$ depends only the tails of $\nu$. We will apply the Corollary~\ref{cor:thin} to the approximations  $\empN_0$ of $F_0$. For these approximations, we  may choose $x_*$ such that
\be
\label{eq:uniformM}
\sup_{n \geq 1} \empN_0([x_*,\infty)) < \veps,
\ee
so that $M$ is independent of $n$. In fact, it may be estimated by the Kolmogorov-Tikhomirov calculation of metric entropy~\cite{Kolmogorov}.

\end{remark}

\section{Concentration of the empirical measure at one point}
\label{sec:conpoint}
The main observation that underlies this section is as follows. Fix $t>0$ and let  $\rho(t) = \int_t^\infty f_0(r)\, dr$. For each even positive integer $n$, set $r = \lfloor \rho(t) n \rfloor$ and assume the empirical measure $\empN_0$ is chosen as in \qref{eq:ic}. As in Section~\ref{subsec:dtoL}, we color the particles $a_1 \ldots a_w$ white, $w=n-r$, and the particles $a_{w+1}, \ldots, a_n$ red. By time $t$ all the white particles have been removed and only a subset  of size $X_{n,r}$ of the red particles remain.  In the urn model we ignored the positions of these particles, and focused only on the number of particles removed. Now we examine their positions more carefully.  Since at each step, particles are removed uniformly, and the particles move at unit speed to the left, at time $t$, $\empN(t)$ is given by an $X_{n,r(t)}$ thinning of the shifted empirical measure with atoms at $a_{w+1},\ldots,a_{n}$. 

We will apply Theorem~\ref{thm:conc-L} and Theorem~\ref{thm:thin} to obtain the following concentration estimate for the deviation from the solution for finite $n$. Recall that the shift operator was defined in \qref{eq:shift1}.

\begin{theorem}
\label{thm:one-point}
There exists a constant $M(\veps,F_0)$ such that for every $t>0$ and every $\veps>0$
\be
\label{eq:thin-conc}
\prob\left(\ddm\left(\empN(t),\rho(t)S_t^*\empN_0\right) > \veps\right) \leq 2(M+1)\exp \left( -\frac{n\veps^2}{256} \right), \quad n \geq n_\veps.
\ee
The parameter $n_\veps$ is defined as in Theorem~\ref{thm:urn}.
\end{theorem}
\begin{proof}   
{\em 1.\/}  In order to simplify the main calculation we will assume that $t$ is chosen so that $r=\rho(t) n$ denote the number of particles in the shifted measure $S_t^*\empN_0$.
In general, $r = \lfloor \rho(t)n\rfloor$, so that $\rho -1/n \leq r/n \leq \rho +1/n$, and the calculations may be modified to included an asymptotically negligible contribution.

\medskip
{\em 2.\/} We define $X_{n,r}$ as in Section~\ref{sec:urn}.  Let $A$ denote the event $\{\ddm\left(\empN(t),\rho(t)S_t^*\empN_0\right) > \veps\}$ and $B$ denote the event $\{ |X_{n,r}/r -\rho| \leq  \veps/2\}$. Then clearly
\begin{align}
\label{eq:cond-prob}    
\prob(A) = \prob(A\left|B\right.) \prob(B) + \prob(A \left|B^c\right.) \prob(B^c) \leq \prob(A\left|B\right.)  + \prob(B^c). 
\end{align}

\medskip
{\em 3.\/} The second term in \qref{eq:cond-prob} is controlled by Theorem~\ref{thm:urn}. Since $r/n=\rho$ and by equation \qref{eq:def-phi}, $\phi(\rho)=\rho^2$ and $\psi(\rho) = 2\rho^2(1-\rho)^2$, we obtain
\be
\label{eq:tail1}        
\prob(B^c) = \prob \left( \left|\frac{X_{n,r}}{r} -\rho\right| > \frac{\veps}{2} \right)\leq 
2\exp\left(-\frac{n\veps^2}{16(1-\rho)^2}\right).
\ee

\medskip
{\em 4.\/} We control the first term on the RHS of \qref{eq:cond-prob} as follows. For each realization in $B$, we have 
\begin{align}
\nn     
\ddm\left(\empN(t),\rho(t)S_t^*\empN_0\right)\leq \ddm\left(\empN(t), \frac{X_{n,r}}{r} S_t^*\empN_0\right) + \ddm \left(\frac{X_{n,r}}{r} S_t^*\empN_0, \rho(t)S_t^*\empN_0\right)\\
\nn
\leq \ddm\left(\empN(t), \frac{X_{n,r}}{r} S_t^*\empN_0\right) + \frac{\veps}{2}.
\end{align}
The last inequality above holds because of the uniform estimate
\be
\left| \left\langle \frac{X_{n,r}}{r} S_t^*\empN_0 - \rho(t)S_t^*\empN_0, \varphi \right\rangle \right| = \left| \frac{X_{n,r}}{r}- \rho(t)\right| \left| \langle S_t^*\empN_0, \varphi \rangle \right| \leq \frac{\veps}{2}, 
\ee
for all realizations in $B$ and $\|\varphi\|_{\BL} \leq 1$.

\medskip
{\em 5.\/} Given $X_{n,r}$ the law of $\empN(t)$ is obtained by a uniform thinning of $S_t^*\empN_0$ from $r$ to $X_{n,r}$ particles. Therefore, we may apply Corollary~\ref{cor:thin} with 
\[ \nu = \frac{n}{r}S_t^*\empN_0,\quad \mu  = \frac{n}{r}\empN(t), \quad s=X_{n,r}, \]
to obtain the uniform estimate
\be
\label{eq:tail2}
\prob(A\left|B\right.) \leq  \prob_{r,s} \left(  \ddm(\mu, \frac{s}{r}\nu) > \frac{\rho \veps}{2} \right)  
\leq 2M\exp\left( -\frac{n \veps^2}{256 \rho} \right).
\ee
Using Remark~\ref{rem:uniformM}, we note that the constant $M$ is independent of $n$ and depends only on the tails of $F_0$. 

\medskip
{\em 6.\/} The rate constants $(1-\rho)^{-2}$ and $\rho^{-1}$ in estimates \qref{eq:tail1} and \qref{eq:tail2} are bounded below by $1$. Thus, the rate constant $1/256$ is a uniform lower bound for the rate constant.
\end{proof}

\section{Uniform concentration of the empirical measure}
\label{sec:unifconc}
We prove Theorem~\ref{thm:conc-emp} in this section. The proof relies on the one-point concentration estimate  from the previous section, and some regularity estimates for the empirical measure. 

\subsection{Regularity estimates for the empirical measure}
\label{subsec:cont}
Recall the shift operator $S_h$ define in \qref{eq:shift1}. Note that even if $\varphi \in \BL$, in general $S_h\varphi$ has a jump at $x=h$. Given a measure $\mu \in \M$ let  $\mu(x) = \mu((0,x))$ denote its distribution function. In order to define a modulus of continuity for $S_h^*$, we introduce 
\be
\label{eq:shift2}
\omega(h;\mu) = \sup_{x \geq 0} \left(\mu(x+h)-\mu(x)\right).
\ee
\begin{lemma}
\label{lemma:lip-cont-mu}
Assume $\mu \in \M$ and $h>0$. Then 
\be
\label{eq:lip-shift}
\ddm\left(\mu, S^*_h \mu\right) \leq \omega(h;\mu) + \mu(\infty)h.
\ee
\end{lemma}
\begin{proof}
Choose a bounded Lipschitz test function $\varphi$ with $\|\varphi\|_{\BL}\leq 1$. We use the definition \qref{eq:shift1} to obtain
\begin{align}
\nn
\langle \mu -S^*_h\mu , \varphi \rangle = \langle \mu, \varphi -S_h\varphi \rangle = \int_0^h \varphi(x) \mu(dx) + \int_h^\infty \left(\varphi(x)-\varphi(x-h) \right) \mu(dx)\\
\nn
\leq \|\varphi\|_\infty \mu(h) + \left(\mathrm{Lip}({\varphi})\right) h \left(\mu(\infty)-\mu(h)\right) \leq \omega(h;\mu) + \mu(\infty) h.
\end{align}
We now take the supremum over all $\varphi$ with $\|\varphi\|_{\BL}\leq 1$ to obtain \qref{eq:lip-shift}. 
\end{proof}
\begin{lemma}
\label{lemma:lip-cont-F}
Assume $f_0\in L_+^1$ and the empirical measures $\empN_0$ is chosen as in \qref{eq:ic}. Then
\be
\label{eq:lip-n}
\omega(h;\empN_0) \leq \omega(h;F_0) + \frac{1}{n}.
\ee
\end{lemma}
\begin{proof}
The choice of empirical measures in \qref{eq:ic} ensures the lower and upper bounds
\be
\label{eq:bound-emp0}
\empN_0(x) -\frac{1}{2n} \leq F_0(x) \leq \empN_0(x) +\frac{1}{2n}, \quad x \in (0,\infty).
\ee
Therefore,
\be
\empN_0(x+h) -\empN_0(x) \leq F_0(x+h) -F_0(x) + \frac{1}{n} \leq \omega(h;F_0) + \frac{1}{n}.
\ee
\end{proof}
\begin{lemma}
\label{lemma:lip-cont-shift}
For each $t>t_0 \geq 0$, the empirical measure for the particle system satisfies the continuity estimate
\be
\label{eq:lip-evolve}
\ddm\left(\empN(t),S^*_{t-t_0}\empN(t_0)\right) \leq  2\left(\lossN(t)-\lossN(t_0)\right).
\ee
\end{lemma}
\begin{proof}
Let $r$ and $s$ denote the number of particles in $\empN(t_0)$ and $\empN(t)$, so that the difference $r-s =2n(\lossN(t)-\lossN(t_0))$. For convenience, let $y_i$, $i=1,\ldots,s$ and $z_j$, $j=1, \ldots,r-s$ denote the particles of $\xx(t_0)$ that are not removed, and  removed by time $t$, respectively. Then
\be
\empN(t_0) = \frac{1}{n} \sum_{i=1}^s \delta_{y_i} +  \frac{1}{n} \sum_{j=1}^{r-s} \delta_{z_j}, \quad  \empN(t) = \frac{1}{n} \sum_{i=1}^s \delta_{y_i-(t-t_0)}.
\ee
Therefore, for each $\varphi \in \BL$ with $\|\varphi\|_{\BL} \leq 1$, 
\be
\left| \langle \empN(t) -S^*_{t-t_0}\empN(t_0), \varphi \rangle \right| \leq  \frac{1}{n} \sum_{i=1}^{r-s} \|\varphi\|_{\infty} \leq 2 \left(\lossN(t)-\lossN(t_0)\right).
\ee
\end{proof}
\begin{lemma}
\label{lemma:rho-cont}
For each $t>t_0 \geq 0$, 
\be
\label{eq:rho-cont}
\left| \rho(t)-\rho(t_0)\right| \leq \omega(t-t_0;F_0).
\ee
\end{lemma}
\begin{proof}
This follows immediately from equation \qref{eq:def-disbn} and \qref{eq:shift2}
\be
\left| \rho(t)-\rho(t_0)\right| \leq \int_{t_0}^t f_0(x) \, dx \leq \omega(t-t_0;F_0).
\ee
\end{proof}

We will combine these estimates to control $\ddm(\empN(t),\rho(t)S_t^*\empN_0)$ on a finite interval $[0,T]$.  To this end, consider 
$T>0$, a positive integer $N$ (fixed; to be chosen later), and 
the uniformly spaced grid
\be
\label{eq:grid}
0=t_0 < t_1 < \ldots < t_N=T, \quad t_{i+1}-t_i =\frac{T}{N}:= h_N, \quad i=0, \ldots, N-1.
\ee
\begin{lemma}
\label{lemma:inter}     
\ba
\label{eq:inter1}
\lefteqn{\sup_{t \in [0,T]} \ddm\left(\empN(t), \rho(t)S_t^*\empN_0\right)}\\
\nn 
&&\leq \max_{0 \leq i\leq N} \ddm\left(\empN(t_i), \rho(t_i)S_{t_i}^*\empN_0\right) + 4 \left( \|\lossN -L\|_{L^\infty} + \omega(h_N;F_0) + h_N + \frac{1}{n}\right). 
\ea
\end{lemma}
\begin{proof}
For $t \in [t_i,t_{i+1})$ we have
\begin{align}
\nn
\ddm\left(\empN(t), \empN(t_i)\right) \leq 
\ddm\left(\empN(t), S_{t-t_i}^*\empN(t_i)\right) + \ddm\left(S_{t-t_i}^*\empN(t_i), \empN(t_i)\right)\\
\nn
\leq 2\left(\lossN(t)-\lossN(t_i)\right) + \omega(h_N;\empN_0) +h_N \\
\label{eq:inter2}
\leq 2\left(\lossN(t_{i+1})-\lossN(t_i)\right) + \omega(h_N;F_0) + h_N+ \frac{1}{n}.
\end{align}
We have used Lemma~\ref{lemma:lip-cont-mu} and Lemma~\ref{lemma:lip-cont-shift} in the first inequality, and Lemma~\ref{lemma:lip-cont-F} in the second. We will control the jumps in the empirical loss as follows
\begin{align}
\nn     
\lossN(t_{i+1})-\lossN(t_i) \leq 2\|\lossN -L\|_{L^\infty} + L(t_{i+1})-L(t_{i}) \\
\nn
\stackrel{\qref{eq:Ldef}}{=} 2\|\lossN -L\|_{L^\infty} + \frac{1}{2}\left(\rho^2(t_{i+1})-\rho^2(t_{i})\right)\\
\label{eq:inter3} \leq 2\|\lossN -L\|_{L^\infty} + \omega(h_N;F_0),
\end{align}
using Lemma~\ref{lemma:rho-cont} in the last step. Finally, we write
\ba
\nn
\lefteqn{\ddm\left(\empN(t), \rho(t)S_t^*\empN_0\right)}\\
\nn
&& \leq         \ddm\left(\empN(t), \empN(t_i)\right) + \ddm\left(\empN(t_i), \rho(t_i)S_{t_i}^*\empN_0\right) + \ddm\left(\rho(t_i)S_{t_i}^*\empN_0, \rho(t)S_t^*\empN_0\right),
\ea
and apply the inequalities above. The first term is controlled by \qref{eq:inter2} and \qref{eq:inter3}. The second term is controlled by taking the maximum over the finite set $0 \leq i \leq N$. The last term is controlled by Lemma~\ref{lemma:rho-cont}.
\end{proof}
\begin{proof}[Proof of Theorem~\ref{thm:conc-emp}]
{\em 1.\/} Let us first separate the effect of difference in initial conditions. Clearly,
\begin{align}
\nn
\ddm \left(\empN(t), \rho(t)S_t^*F_0\right) \leq \ddm\left(\empN(t), \rho(t)S_t^*\empN_0\right)  + \ddm \left( \rho(t)S_t^*\empN_0, \rho(t)S_t^*F_0\right)\\
\nn
\leq \ddm\left(\empN(t), \rho(t)S_t^*\empN_0\right) + \ddm\left(\empN_0, F_0\right).
\end{align}
We increase $n_\veps$ if necessary, so that $\ddm\left(\empN_0, F_0\right)< \veps$ for $n \geq n_\veps$.

\medskip
{\em 2.\/} Given $\veps>0$ we may choose $N$ sufficiently large that $\omega(h_n;F_0)+h_N <\veps/2$. We may further increase $n_\veps$ if necessary so that $n_\veps^{-1} < \veps/2$. Lemma~\ref{lemma:inter} then implies that 
\be
\label{eq:inter4}
\sup_{t \in [0,T]} \ddm\left(\empN(t), \rho(t)S_t^*\empN_0\right) \leq  \max_{0 \leq i\leq N} \ddm\left(\empN(t_i), \rho(t_i)S_{t_i}^*\empN_0\right) + 4 \|\lossN -L\|_{L^\infty} + \veps. 
\ee

\medskip
{\em 3.\/} We combine steps 1 and 2 to obtain
\ba
\nn
\lefteqn{\prob\left(  \sup_{t \in [0,T]} \ddm\left(\empN(t), \rho(t)S_t^*F_0\right) >3 \veps \right) \leq \prob\left(  \sup_{t \in [0,T]} \ddm\left(\empN(t), \rho(t)S_t^*\empN_0\right) >2 \veps \right) } \\
\nn
&& \leq \prob\left( \max_{0 \leq i\leq N} \ddm\left(\empN(t_i), \rho(t_i)S_{t_i}^*\empN_0\right) > \frac{\veps}{2}\right) + \prob\left( 4 \|\lossN -L\|_{L^\infty} > \frac{\veps}{2} \right) \\
\nn
&& \leq \sum_{i=1}^N \prob\left( \ddm\left(\empN(t_i), \rho(t_i)S_{t_i}^*\empN_0\right) > \frac{\veps}{2}\right) + \prob\left(\|\lossN -L\|_{L^\infty} > \frac{\veps}{8} \right).
\ea
The probability that $\|\lossN -L\|_{L^\infty} >\veps/8$ is controlled by Theorem~\ref{thm:conc-L}, and the probability that $\ddm\left(\empN(t_i), \rho(t_i)S_{t_i}^*\empN_0\right) >\veps/2$ at each $t_i$ is controlled by Theorem~\ref{thm:one-point}.
\end{proof}

\section{Acknowledgements}
This work was supported by NSF grant DMS 14-11278. The authors are grateful to Bob Pego for several discussions regarding~\cite{JK}. GM thanks the International Center for Theoretical Sciences (ICTS) and the Raman Research Institute (RRI) in Bangalore for their hospitality during the preparation of this manuscript.


\bibliographystyle{siam}
\bibliography{km}

\end{document}